\documentclass{amsart}

\newtheorem{theorem}{Theorem}[section]
\newtheorem{lemma}[theorem]{Lemma}

\theoremstyle{definition}
\newtheorem{definition}[theorem]{Definition}

\theoremstyle{remark}
\newtheorem{remark}[theorem]{Remark}

\numberwithin{equation}{section}
\def\R{\mathbb{R}}

\begin{document}

\title[A proof of the fundamental theorem]
 {A proof of the fundamental theorem of curves in space and its applications.}

\author{H\'ector Efr\'en Guerrero Mora}
\address{Department of Mathematics, Universidad del Cauca, Cauca, Colombia}
\email{heguerrero@unicauca.edu.co}
\thanks{The first author was supported in part by Universidad del Cauca project ID 4558.}


\subjclass[2000]{Primary 53A04; Secondary 53A55}
\keywords{fundamental theorem of curves in space, general helices, slant helices}
\date{december 7, 2018}
\dedicatory{To my family.}

\begin{abstract}
We give a necessary and suficente condition for the existence of a space curve with curvature $\kappa$ and torsion $\tau$ finding a solution of a nonlinear differential equation of second order and some applications are given for the general helices and slant helices.
\end{abstract}

\maketitle
\section{The fundamental theorem.}
The curves parameterized in space are objects of great interest and with them we can model and analyze problems that appear in different fields of physics and other sciences.\\ In the classical differential geometry of curves it is well known that there are two geometric invariants that are: curvature and torsion; which determine the behavior of the curve in space. The fundamental theorem of the local theory of curves in space states that if two differentiable $\kappa=\kappa(s)$ and $\tau=\tau(s)$ functions are given and positive  $\kappa=\kappa(s)$ shows that there is a curve whose curvature is $\kappa=\kappa(s)$ and whose torsion is  $\tau=\tau(s)$ and any other curve that meets the conditions that its curvature is $\kappa=\kappa(s)$ and its torsion is $\tau=\tau(s)$ differs from the previous curve by a rigid movement.\\
The demonstration of this result, which appears in most texts of differential geometry, is done considering a system of nine differential equations and using the existence and uniqueness theorem of ordinary differential equations.\\
In this paper we show a new proof of \textbf{the theorem of the local theory of curves in space} \cite{DoC:76}.
\\ Considering a differential equation of non-linear second order:
\begin{equation*}
\frac{d}{ds}\{\frac{1}{\kappa}\frac{d\omega}{ds}\}=-\kappa \omega+\tau\sqrt{1-\omega^2-(\frac{1}{\kappa}\frac{d\omega}{ds})^2}.
\end{equation*}
We shall prove the fundamental theorem in the form
\begin{theorem}
Given differentiable function $\kappa(s)>0$ and the continous function $\tau(s)$, $s\in (a,b)$, there exists a regular parametrized curve $\alpha:J\rightarrow \R^3$ such that $s$ is the arc length, $\kappa(s)$ is the curvature, and $\tau(s)$ is the torsion of $\alpha.$ Moreover, any other curve $\widetilde{\alpha}$, satisfying the same conditions, differs from $\alpha$ by a rigid motion; that is, there exists an orthogonal linear map $\rho$ of $\R^3,$ with positive determinant, and a vector $c$ such that $\widetilde{\alpha}=\rho\circ\alpha +c.$
\end{theorem}
In books on elementary differential geometry (see for example \cite{KuWo:06}, the condition for function $\tau$ is usually given differentiable.\\
The new proof we give of this theorem is based on a theorem of existence and uniqueness for a nonlinear differential equation of second orden. \ref{teorema de existencia y unicidad}
\section{Proof of the fundamental theorem.}
\begin{lemma}\label{teorema de existencia y unicidad}
Let $\kappa:[c,d]\rightarrow \R$ be a function always positive of class $C^1$, and\\ $\tau:[c,d]\rightarrow \R$ be a function of class $C^0$.Then the second-order differential equation
\begin{equation}\label{Ecuacion}
\frac{d}{ds}\{\frac{1}{\kappa}\frac{dw}{ds}\}=-\kappa w+\tau\sqrt{1-w^2-(\frac{1}{\kappa}\frac{dw}{ds})^2},
\end{equation}with initial value \begin{eqnarray*}
w(s_1)&=&w_1\\
w'(s_1)&=&v_1,
\end{eqnarray*}where $s_1\in (c,d)$ and $(w_1,v_1)\in \{(w,v)\in \R^2\mid w^2+\frac{v^2}{\kappa_0^2}<1\}$, and \\$\kappa_0=\text{min}\{\kappa(s)\mid s\in [c,d]\}$, has a unique solution on some interval open $J\subset (c,d)$ containing $s_1$.
\end{lemma}
\begin{proof}
Writing the equation \ref{Ecuacion} as:
\begin{equation*}
w''=\frac{\kappa'}{\kappa}w'-\kappa^2w+\tau\sqrt{\kappa^2(1-w^2)-(w')^2}
\end{equation*}
 and introducing the dependent variable, $v=w'$.
It follows that the equation can now be rewritten in the form
\begin{equation*}
(w,v)'=( v,\frac{\kappa'}{\kappa}v-\kappa^2w+\tau\sqrt{\kappa^2(1-w^2)-v^2}).
\end{equation*}Note that the vector field
\begin{equation*}
F(s,w,v)=( v,\frac{\kappa'}{\kappa}v-\kappa^2w+\tau\sqrt{\kappa^2(1-w^2)-v^2}),
\end{equation*} where $(s,w,v)\in\mathfrak{I}=(c,d)\times\{(w,v)\in \R^2\mid w^2+\frac{v^2}{\kappa_0^2}<1\}$ is a continous fuction and Lipschitziana with respect to $(w,v)$ in a neighborhood $D$ of\\ $(s_1,w_1,v_1)\in \mathfrak{I}.$
\\In fact, the function $F$ is well defined, since $ w^2+\frac{v^2}{\kappa_0^2}<1$ implies\\ $0<\kappa^2(1-w^2)-v^2$, and $F$ is continous;
and it is clear that the
partial derivates \begin{equation*}\frac{\partial F_1}{\partial w},\frac{\partial F_1}{\partial v},\frac{\partial F_2}{\partial w}\ \text{and} \ \frac{\partial F_2}{\partial v} \end{equation*} they are continuous, since
\begin{eqnarray*}
\frac{\partial F_1}{\partial w}(s,w,v)&=&0 \\
\frac{\partial F_1}{\partial v}(s,w,v)&=&1 \\
\frac{\partial F_2}{\partial w}(s,w,v)&=&-\kappa^2-\frac{\tau\kappa^2w}{\sqrt{\kappa^2(1-w^2)-v^2}}\\
\frac{\partial F_2}{\partial v}(s,w,v)&=&\frac{\kappa'}{\kappa}-\frac{\tau v }{\sqrt{\kappa^2(1-w^2)-v^2}},
\end{eqnarray*}therefore $F$ is continous and Lipschitziana with respect to $(w,v)$ in a neighborhood $D$ of $(s_1,w_1,v_1)\in \mathfrak{I}.$\\
And this implies by the $Picard-Lindel\ddot{o}f 's\ theorem$ that the problem has a unique solution on some interval open $J\subset (c,d)$ containing $s_1$.
\end{proof}
\begin{theorem}
Given differentiable function $\kappa(s)>0$ and the continous function $\tau(s)$, $s\in (a,b)$, there exists a regular parametrized curve $\alpha:J\rightarrow \R^3$ such that $s$ is the arc length, $\kappa(s)$ is the curvature, and $\tau(s)$ is the torsion of $\alpha.$ Moreover, any other curve $\widetilde{\alpha}$, satisfying the same conditions, differs from $\alpha$ by a rigid motion; that is, there exists an orthogonal linear map $\rho$ of $\R^3,$ with positive determinant, and a vector $c$ such that $\widetilde{\alpha}=\rho\circ\alpha +c.$
\end{theorem}
\begin{proof}
Let's find an $\alpha$ curve, parameterized by arc length $s$, such that its curvature $\kappa_{\alpha}$ is equal to $\kappa$ and its torsion $\tau_{\alpha}$ is equal to $\tau.$
Let's write its tangent vector $\textbf{t}=\textbf{t}(s)$ in polar coordinates,
\begin{equation*}
\textbf{t} =(\sin \phi\cos \theta,\sin \phi\sin \theta,\cos \phi).
\end{equation*}Therefore, its normal vector $\textbf{n}=\textbf{n}(s)$ and its binormal vector $\textbf{b}=\textbf{b}(s)$ are given by
\begin{eqnarray*}\textbf{n}&=&
(\frac{\phi'\cos \phi\cos \theta-\theta'\sin \phi\sin \theta}{\kappa},\frac{ \phi'\cos \phi\sin \theta +\theta'\sin\phi\cos \theta}{\kappa},\frac{-\phi'\sin \phi}{\kappa})
\\ \textbf{b}&=&(\frac{-\phi'\sin \theta}{\kappa}-\frac{\theta'\sin 2\phi\cos\theta}{2\kappa},\frac{-\phi'\cos \theta}{\kappa}-\frac{\theta'\sin 2\phi\sin \theta}{2\kappa},\frac{\theta'\sin^2 \phi}{\kappa}).\end{eqnarray*}
It is known that its Frenet trihedron $\textbf{t},\textbf{n},\textbf{b}$ forms an orthonormal basis of $\R^3$ and satisfies
\begin{eqnarray*}
\frac{d\textbf{t}}{ds}&=&\kappa \textbf{n}\\
\frac{d\textbf{n}}{ds}&=&-\kappa\textbf{t}+\tau\textbf{b}\\
\frac{d\textbf{b}}{ds}&=&-\tau \textbf{n}.
\end{eqnarray*}Therefore, for $w=<\textbf{t},D>$,where $D$ is a fixed unit vector, we have.
\begin{equation*}
\frac{dw}{ds}=<\frac{d\textbf{t}}{ds},D>=<\kappa \textbf{n},D>=\kappa<\textbf{n},D>,
\end{equation*} and $\displaystyle\frac{1}{\kappa}\frac{dw}{ds}=<\textbf{n},D>.$\ This implies
\begin{eqnarray*}
\frac{d}{ds}\{\frac{1}{\kappa}\frac{dw}{ds}\}&=&<\frac{d\textbf{n}}{ds},D>=<-\kappa\textbf{t}+\tau\textbf{b},D>\\&=&
-\kappa<\textbf{t},D>+\tau<\textbf{b},D>\\&=&-\kappa <\textbf{t},D>\pm\tau\sqrt{1-<\textbf{t},D>^2-<\textbf{n},D>^2}\\&=&-\kappa w\pm\tau\sqrt{1-w^2-(\frac{1}{\kappa}\frac{dw}{ds})^2}.
\end{eqnarray*}
Since for any curve $\alpha$ there exists an orthogonal linear function $\sigma$ of $\R^3$, with positive determinant such that the binormal vector $\textbf{b}$ of $\sigma\circ\alpha$ satisfies $<\textbf{b},(0,0,1)>0$ in an neighborhood of $s_0\in (a,b)$
and knowing that the curvature $\kappa$ and the torsion $\tau$ are invariant given a rigid movement. We can assume from the beginning that $D=(0,0,1)$ and the binormal vector $\textbf{b}$ of $\alpha$ satisfies $<\textbf{b},(0,0,1)>0$ in an neighborhood of $s_0\in (a,b)$.\\
Therefore, we can consider the initial value problem:
\begin{equation}\label{Ecuacion de segundo orden}
\frac{d}{ds}\{\frac{1}{\kappa}\frac{dw}{ds}\}=-\kappa w+\tau\sqrt{1-w^2-(\frac{1}{\kappa}\frac{dw}{ds})^2},
\end{equation} \begin{eqnarray*}
w(s_0)&=&w_0,\\
w'(s_0)&=&v_0,
\end{eqnarray*}where $s_0\in (a+\epsilon,b-\epsilon)$, $(w_0,v_0)\in\{(w,v)\in \R^2\mid w^2+\frac{v^2}{\kappa_0^2}<1\}$, and \\$\kappa_0=\text{min}\{\kappa(s)\mid s\in [a+\epsilon,b-\epsilon]\}$, for some $\epsilon>0$.\\
Which by Lemma \ref{teorema de existencia y unicidad}, has a unique solution $\xi=\xi(s)$ on some interval open $J\subset (a,b)$ containing $s_0$.
\\Now, given that $\cos \phi=<\textbf{t},(0,0,1)>$ and $\displaystyle\frac{\theta'\sin^2 \phi}{\kappa}=<\textbf{b},(0,0,1)>,$
 we have that there are functions $\widetilde{\phi}$ and $\widetilde{\theta}$, such that
\begin{eqnarray*}
\widetilde{\phi}&=&\arccos{<\textbf{t},(0,0,1)>}=\arccos{\xi},\\
\widetilde{\theta}&=&\int<\textbf{b},(0,0,1)>ds=\int{\frac{\kappa \sqrt{1-\xi^2-(\frac{1}{\kappa}\frac{d\xi}{ds})^2}}{\sin^2\widetilde{\phi}}ds}
\\&=&\int{\frac{\kappa\sqrt{1-\xi^2-(\frac{1}{\kappa}\frac{d\xi}{ds})^2}}{1-\xi^2}ds}
\end{eqnarray*}
By replacing these functions in the tangent vector $\textbf{t}$, given in polar coordinates, we find an $\alpha$ curve, given by $\alpha(s)=(x(s),y(s),z(s))$, where
\begin{eqnarray*}
x(s)&=&\int{\sqrt{1-\xi^2}\cos(\int{\frac{\kappa \sqrt{1-\xi^2-(\frac{1}{\kappa}\frac{d\xi}{ds})^2}}{1-\xi^2}ds})ds}\\
y(s)&=&\int{\sqrt{1-\xi^2}\sin(\int{\frac{\kappa\sqrt{1-\xi^2-(\frac{1}{\kappa}\frac{d\xi}{ds})^2}}{1-\xi^2}}ds})ds\\
z(s)&=&\int {\xi} ds.
\end{eqnarray*}This curve is parameterized by arc length $s$, its curvature is $\kappa_{\alpha}=\kappa$ and its torsion is $\tau_{\alpha}=\tau$.
In effect, using the curvature formula and the torsion formula, we have
\begin{eqnarray*}
\kappa_{\alpha}&=&\mid\mid \alpha''\mid \mid=\sqrt{\frac{\xi^2(\xi')^2}{1-\xi^2}+\frac{\kappa^2(1-\xi^2-(\frac{1}{\kappa}\frac{d\xi}{ds})^2)}{1-\xi^2}+(\xi')^2}=\kappa\\
\tau_{\alpha}&=&\frac{\alpha'\wedge \alpha''\cdot \alpha'''}{\mid\mid\alpha'\wedge\alpha''\mid \mid^2}=\frac{\xi\kappa^3-\xi'\kappa'+\kappa\xi''}{\kappa\sqrt{(1-\xi^2)\kappa^2-(\xi')^2}}=\frac{\frac{d}{ds}\{\frac{1}{\kappa}\frac{d\xi}{ds}\}+\kappa \xi}{\sqrt{1-\xi^2-(\frac{1}{\kappa}\frac{d\xi}{ds})^2}}=\tau.
\end{eqnarray*}
Now, suppose that $\alpha=\alpha(s)$ is a curve parametrized by arc lenth $s$, where\\ $\kappa=\kappa(s)$ is its curvature, $\tau=\tau(s)$ is its torsion and consider the canonical basis $\{\textbf{e}_i\mid i=1,2,3\}$ of $R^3.$  And let $\textbf{t}_{0},\textbf{n}_{0},\textbf{b}_{0}$ be the Frenet trihedron at $s=s_0\in I$ of $\alpha$. Clearly, there exists an orthogonal linear map $\sigma$ of $R^3,$ with positive determinant such that the scalar product
$<\sigma\circ\textbf{b}_{0},\textbf{e}_i>$ is greater than zero, for all $i=1,2,3.$\\ To each value of the parameter $s$ of the curve $\beta=\sigma\circ \alpha+c^*$, we have associated the Frenet trihedron at $s$:\begin{equation*}\textbf{t}_{\beta}=\textbf{t}_{\beta}(s),\textbf{n}_{\beta}=\textbf{n}_{\beta}(s),\textbf{b}_{\beta}=\textbf{b}_{\beta}(s).
\end{equation*}Thus
Frenet trihedron forms an orthonormal basis of $\R^3$ and satisfies:
\begin{equation*}
1=<\textbf{t}_{\beta},\textbf{e}_i>^2+<\textbf{n}_{\beta},\textbf{e}_i>^2+<\textbf{b}_{\beta},\textbf{e}_i>^2 ,
\end{equation*}for all $i=1,2,3$ and
\begin{eqnarray*}
\frac{d\textbf{t}_{\beta}}{ds}&=&\kappa \textbf{n}_{\beta}\\
\frac{d\textbf{n}_{\beta}}{ds}&=&-\kappa\textbf{t}_{\beta}+\tau\textbf{b}_{\beta}\\
\frac{d\textbf{b}_{\beta}}{ds}&=&-\tau \textbf{n}_{\beta}.
\end{eqnarray*}
It is clear, because of the continuity of the $\beta$ curve, that exists a neighborhood $N$ of $s_0\in I$ such that the binormal vector $\textbf{b}_{\beta}=\textbf{b}_{\beta}(s)$ of the $\beta$ curve satisfies that its scalar product $<\textbf{b}_{\beta},e_i>$ is greater than zero, for all $i=1,2,3.$
Hence
\begin{eqnarray*}
\frac{d}{ds}<\textbf{n}_{\beta},\textbf{e}_i>&=&<\frac{d\textbf{n}_{\beta}}{ds},\textbf{e}_i>\\
&=&<-\kappa\textbf{t}_{\beta}+\tau\textbf{b}_{\beta},\textbf{e}_i>\\
&=&-\kappa<\textbf{t}_{\beta},\textbf{e}_i>+\tau<\textbf{b}_{\beta},\textbf{e}_i>,
 \end{eqnarray*}this is
\begin{equation*}
\frac{d}{ds}(\frac{1}{\kappa}\frac{d}{ ds}<\textbf{t}_{\beta},\textbf{e}_i>)=-\kappa<\textbf{t}_{\beta},\textbf{e}_i>+\tau\sqrt{1-<\textbf{t}_{\beta},\textbf{e}_i>^2-(\frac{1}{\kappa}\frac{d}{ ds}<\textbf{t}_{\beta},\textbf{e}_i>)^2}.
\end{equation*}
This implies that the components $<\textbf{t}_{\beta},\textbf{e}_i>$ of the tangent vector $\textbf{t}_{\beta}$ of the curve $\beta$ satisfies the initial value problem
\begin{eqnarray}\label{Problemas de valores iniciales}
\frac{d}{ds}\{\frac{1}{\kappa}\frac{dw}{ds}\}&=&-\kappa w+\tau\sqrt{1-w^2-(\frac{1}{\kappa}\frac{dw}{ds})^2}\\
w(s_0)&=&<\sigma\circ \textbf{t}_0,\textbf{e}_i>\\
w'(s_0)&=&<\kappa(s_0)\sigma\circ \textbf{n}_0,\textbf{e}_i>,
\end{eqnarray} in some neighborhoods $I_i\subset N$, of $s_0\in I$, for $i=1,2,3$, respectively.\\
Now, suppose there is another $\widetilde{\alpha}$ curve such that its curvature $\widetilde{\kappa}(s)$ is equal to $\kappa(s)$ and its torsion $\widetilde{\tau}(s)$ is equal to $\tau(s)$, where $s\in I.$\\ Let $\widetilde{\textbf{t}}_{0},\widetilde{\textbf{n}}_{0},\widetilde{\textbf{b}}_{0}$ be the Frenet trihedron at $s=s_0\in I$ of $\widetilde{\alpha}$. Then, there exists an orthogonal linear map $\pi$ of $\R^3,$ with positive determinant such that
\begin{eqnarray*}
\pi\circ \widetilde{\textbf{t}}_{0}&=&\sigma\circ\textbf{t}_0\\
\pi\circ \widetilde{\textbf{n}}_{0}&=&\sigma\circ\textbf{n}_0\\
\pi\circ \widetilde{\textbf{b}}_{0}&=&\sigma\circ\textbf{b}_0.
\end{eqnarray*}
And consider the  Frenet trihedron $\textbf{t}_{\gamma},\textbf{n}_{\gamma},\textbf{b}_{\gamma}$ of the curve $\gamma=\pi\circ \widetilde{\alpha}+\widetilde{c}.$
Therefore, there is a neighborhood $M$ of $s_0\in I$ such that the binormal vector $\textbf{b}_{\gamma}=\textbf{b}_{\gamma}(s)$ of the $\gamma$ curve satisfies that its scalar product $<\textbf{b}_{\gamma},\textbf{e}_i>$ is greater than zero, for all $i=1,2,3.$
This implies that the components $<\textbf{t}_{\gamma},\textbf{e}_i>$ of the tangent vector $\textbf{t}_{\gamma}$ of the curve $\gamma$ satisfies the initial value problem \ref{Problemas de valores iniciales}, in some neighborhoods $J_i\subset M$, of $s_0\in I$, for $i=1,2,3$, respectively.\\
Therefore, since the solution to the initial value problem is unique, we have to
\begin{equation*}
\textbf{t}_{\beta}(s)=\textbf{t}_{\gamma}(s),
\end{equation*}or all $s\in (\bigcap_{i=1}^3I_i)\cap (\bigcap_{i=1}^3J_i).$
Therefore there exists an orthogonal linear map $\rho$ of $\R^3,$ with positive determinant, and a vector $c$ such that $\widetilde{\alpha}(s)=\rho\circ\alpha(s) +c,$ for all $s\in (\bigcap_{i=1}^3I_i)\cap (\bigcap_{i=1}^3J_i).$
\end{proof}
\subsection{Some observations}

\begin{remark}
Let $\kappa:[a,b]\rightarrow \R$ be a function always positive of class $C^1$, and\\ $\tau:[a,b]\rightarrow \R$ be a function of class $C^0$.
\\If $\xi=\xi(s)$ is a solution of
\begin{equation*}
\frac{d}{ds}\{\frac{1}{\kappa}\frac{dw}{ds}\}=-\kappa w+\tau\sqrt{1-w^2-(\frac{1}{\kappa}\frac{dw}{ds})^2},
\end{equation*} \begin{eqnarray*}
w(s_0)&=&w_0,\\
w'(s_0)&=&v_0,
\end{eqnarray*}where $s_0\in (a+\epsilon,b-\epsilon)$, $(w_0,v_0)\in\{(w,v)\in \R^2\mid w^2+\frac{v^2}{\kappa_0^2}<1\}$, and \\$\kappa_0=\text{min}\{\kappa(s)\mid s\in [a+\epsilon,b-\epsilon]\}$, for some $\epsilon>0$.\\then $\alpha(s)=(x(s),y(s),z(s))$, where
\begin{eqnarray*}
x(s)&=&\int{\sqrt{1-\xi^2}\cos(\int{\frac{\sqrt{(1-\xi^2)\kappa^2-(\xi')^2}}{1-\xi^2}ds})ds}\\
y(s)&=&\int{\sqrt{1-\xi^2}\sin(\int{\frac{\sqrt{(1-\xi^2)\kappa^2-(\xi')^2}}{1-\xi^2}ds})ds}\\
z(s)&=&\int {\xi} ds,
\end{eqnarray*}is a curve parametrized by arc length $s$, $\kappa=\kappa(s)$ is the curvature, and $\tau=\tau(s)$ is the torsion of $\alpha$.\\
Reciprocally, let $\alpha:I\rightarrow \R^3$ be a curve parametrized by arc lenth $s$, where $\kappa=\kappa(s)$ is its curvature and $\tau=\tau(s)$ is its torsion and let $\textbf{t}_{0},\textbf{n}_{0},\textbf{b}_{0}$ be the Frenet trihedron at $s=s_0\in I$ of $\alpha$ and consider the canonical basis $\{\textbf{e}_i\mid i=1,2,3\}$ of $R^3.$\\  Then, there exists an orthogonal linear map $\sigma$ of $\R^3,$ with positive determinant such that the components $<\textbf{t}_{\beta}(s),\textbf{e}_i>$ of the tangent vector $\textbf{t}_{\beta}$ of the curve $\beta=\sigma\circ \alpha$ satisfies the initial value problem
\begin{eqnarray*}
\frac{d}{ds}\{\frac{1}{\kappa}\frac{dw}{ds}\}&=&-\kappa w+\tau\sqrt{1-w^2-(\frac{1}{\kappa}\frac{dw}{ds})^2}\\
w(s_0)&=&<\sigma\circ \textbf{t}_0,\textbf{e}_i>\\
w'(s_0)&=&<\kappa(s_0)\sigma\circ \textbf{n}_0,\textbf{e}_i>,
\end{eqnarray*} in some neighborhoods $I_i$ of $s_0\in I$, for $i=1,2,3$, respectively
\end{remark}
\section{Applications}
\begin{definition}
A curve $\alpha$ with $\kappa(s)\neq 0$ is called a general helix if the principal tangent lines of $\alpha$ make a constant angle with a fixed direction. \cite{KuWo:06}
\end{definition}
\begin{theorem}
Let $\alpha$ be a unit speed space curve with $\kappa=\kappa(s)\neq 0$ and torsion $\tau=\tau(s).$ Then the following statements are equivalent:
\begin{enumerate}
\item  $\alpha$ is a general helix.
\item \begin{equation*}
 \frac{\tau}{\kappa}(s)=m , \  \text{(constant)}
 \end{equation*}
 \item $\alpha(s)=(x(s),y(s),z(s))$, where
 \begin{eqnarray*}
  x(s)&=&\frac{1}{\sqrt{1+m^2}}\int{\cos(\sqrt{1+m^2}\int{\kappa ds})}ds,\\
  y(s)&=&\frac{1}{\sqrt{1+m^2}}\int{\sin(\sqrt{1+m^2}\int{\kappa ds})}ds,\\
  z(s)&=&\frac{m s}{\sqrt{1+m^2}}.
 \end{eqnarray*}Any other curve, satisfying the same conditions, differs from $\alpha$ by a rigid movement.
\end{enumerate}
\end{theorem}
\begin{proof}
Assume property (1) holds, this is suppose that the curve $\alpha:I\rightarrow \R^3$, with curvature $\kappa(s)\neq 0$ and torsion $\tau(s)$ is an general helix and consider its Frenet trihedron $\textbf{t},\textbf{n},\textbf{b}$. Then the principal tangent lines of $\alpha$ make a constant angle with a fixed direction, this is there is a fixed unit vector $U$, such that\\ $<\textbf{t},U>=\delta$ (constant).
We know if $w=<\textbf{t},D>$, then
\begin{equation*}
\frac{d}{ds}\{\frac{1}{\kappa}\frac{dw}{ds}\}=-\kappa w\pm\tau\sqrt{1-w^2-(\frac{1}{\kappa}\frac{dw}{ds})^2},
\end{equation*} therefore, taking $D=U$, we have to
\begin{equation*}
0=-\kappa \delta\pm\tau\sqrt{1-\delta^2},
\end{equation*}this is
\begin{equation*}
\frac{\tau}{\kappa}=\pm\frac{\delta}{\sqrt{1-\delta^2}}
\end{equation*}
Now assume property (2) holds, this is consider the equation
\begin{equation*}
\frac{\tau}{\kappa}(s)=m, \text{(constant)}
\end{equation*}
Now replacing $\tau=m\kappa$ in the second order differential equation \ref{Ecuacion de segundo orden}, we have
\begin{equation*}
\frac{d}{ds}\{\frac{1}{\kappa}\frac{dw}{ds}\}=-\kappa w+\kappa m\sqrt{1-w^2-(\frac{1}{\kappa}\frac{dw}{ds})^2}.
\end{equation*}Considering the definition of the general helix, let's find a solution of the form $\xi(s)=\delta$, where $\delta$ is an constant and satisfies $1-\delta^2>0$.\\
Replacing in the equation above we find that
\begin{equation*}
0=-\kappa \delta +\kappa m\sqrt{1-\delta^2 }.\\
\end{equation*}This implies that $ \delta = \frac{m}{\sqrt{1+m^2}}$.\\
Therefore, $\xi(s)=\frac{m}{\sqrt{1+m^2}},$  is a solution of the non-linear second-order differential equation.
\begin{equation*}
\frac{d}{ds}\{\frac{1}{\kappa}\frac{dw}{ds}\}=-\kappa w+\kappa m\sqrt{1-w^2-(\frac{1}{\kappa}\frac{dw}{ds})^2},
\end{equation*}\\Let's calculate the  position vector of $\alpha$ using the expression
\begin{eqnarray*}
x(s)&=&\int{\sqrt{1-\xi^2}\cos(\int{\frac{\kappa\sqrt{1-\xi^2-(\frac{1}{\kappa}\frac{d\xi}{ds})^2}}{1-\xi^2}ds})ds}\\
y(s)&=&\int{\sqrt{1-\xi^2}\sin(\int{\frac{\kappa\sqrt{1-\xi^2-(\frac{1}{\kappa}\frac{d\xi}{ds})^2}}{1-\xi^2}ds})ds}\\
z(s)&=&\int {\xi} ds.
\end{eqnarray*}In effect,
\begin{equation*}
\int{\frac{\kappa\sqrt{1-\xi^2-(\frac{1}{\kappa}\frac{d\xi}{ds})^2}}{1-\xi^2}}ds=\sqrt{1+m^2}\int{\kappa ds}
\end{equation*}
Therefore, the components of the position vector are:
\begin{eqnarray*}
  x(s)&=&\frac{1}{\sqrt{1+m^2}}\int{\cos(\sqrt{1+m^2}\int{\kappa ds})}ds,\\
  y(s)&=&\frac{1}{\sqrt{1+m^2}}\int{\sin(\sqrt{1+m^2}\int{\kappa ds})}ds,\\
  z(s)&=&\frac{m s}{\sqrt{1+m^2}}
 \end{eqnarray*}
Now assume property (3) holds, a direct calculation shows that its tangent vector $\textbf{t}_{\alpha}$ satisfies
\begin{equation*}
<\textbf{t}_{\alpha}(s),(0,0,1)>=\frac{ m}{\sqrt{1+m^2}},
\end{equation*}this is the principal tangent lines of $\alpha$ make a constant angle with a fixed direction $U=(0,0,1).$\\
Now, if there is another curve that differs from $\alpha$ by a rigid movement, then it's tangent vector is given by $\rho \textbf{t}_{\alpha}$, where $\rho$ is an orthogonal linear map of $\R^3,$ with positive determinant. For this case we take as a fixed direction $\rho (0,0,1.)$
\end{proof}
\begin{definition}
A curve $\alpha$ with $\kappa(s)\neq 0$ is called a slant helix if the principal normal lines of $\alpha$ make a constant angle with a fixed direction. \cite{IzTa:04}
\end{definition}
\begin{theorem}
Let $\alpha$ be a unit speed space curve with $\kappa=\kappa(s)\neq 0$ and torsion $\tau=\tau(s).$ Then the following statements are equivalent:
\begin{enumerate}
\item  $\alpha$ is a slant helix.
\item \begin{equation*}
 \sigma(s)=(\frac{\kappa^2}{(\kappa^2+\tau^2)^{3/2}}(\frac{\tau}{\kappa})')(s)
 \end{equation*}is a constant function.
 \item $\alpha(s)=(x(s),y(s),z(s))$, where
\begin{eqnarray*}
x(s)&=&\frac{1}{\sqrt{1+m^2}}\int(\int{\sin[\frac{\sqrt{1+m^2}\arccos(m\int^s_0\kappa ds)}{m}]\kappa(s)}ds)ds\\
y(s)&=&\frac{1}{\sqrt{1+m^2}}\int(\int{\cos[\frac{\sqrt{1+m^2}\arccos(m\int^s_0\kappa ds)}{m}]\kappa(s)}ds)ds\\
z(s)&=&\frac{\mid m\mid}{\sqrt{1+m^2}}\int(\int^s_0\kappa ds)ds.
\end{eqnarray*}
\end{enumerate}
\end{theorem}
\begin{proof}
Assume property (1) holds, this is suppose that the curve $\alpha:I\rightarrow \R^3$, with curvature $\kappa(s)\neq 0$ and torsion $\tau(s)$ is an slant helix and consider its Frenet trihedron $\textbf{t},\textbf{n},\textbf{b}$. Then the principal normal lines of $\alpha$ make a constant angle with a fixed direction, this is there is a fixed unit vector $U$, such that $<\textbf{n},U>=\delta$ (constant). \\We know if $w=<\textbf{t},D>$, then
\begin{equation*}
\frac{d}{ds}\{\frac{1}{\kappa}\frac{dw}{ds}\}=-\kappa w\pm\tau\sqrt{1-w^2-(\frac{1}{\kappa}\frac{dw}{ds})^2},
\end{equation*} therefore, taking $D=U$, we have to
\begin{equation*}
\kappa w=\pm\tau\sqrt{1-w^2-\delta^2},
\end{equation*}therefore $w^2$ is given by
\begin{equation*}
w^2=\frac{(1-\delta^2)(\frac{\tau}{\kappa})^2}{1+(\frac{\tau}{\kappa})^2}.
\end{equation*}
Deriving with respect to $s$ and since $w$ is equal to $\delta\int^s_{0}\kappa$, we have to
\begin{equation*}
\frac{\delta}{\sqrt{1-\delta^2}}=\frac{(\frac{\tau}{\kappa})'}{\kappa(1+(\frac{\tau}{\kappa})^2)^{3/2}}.
\end{equation*}
Now assume property (2) holds, this is consider the equation
\begin{equation*}
\frac{\kappa^2}{(\kappa^2+\tau^2)^{3/2}}(\frac{\tau}{\kappa})'=
\frac{(\frac{\tau}{\kappa})'}{\kappa(1+(\frac{\tau}{\kappa})^2)^{3/2}}=m\end{equation*}
The torsion $\tau$ can be expressed as:
\begin{equation*}
\tau=\kappa(\frac{(m\int^s_{0}\kappa ds+A)^2}{1-(m\int^s_{0}\kappa ds+A)^2})^{1/2},
\end{equation*}where $A$ is an integration constant.
Now replacing $\tau$ in the second order differential equation \ref{Ecuacion de segundo orden}, we have
\begin{equation*}
\frac{d}{ds}\{\frac{1}{\kappa}\frac{dw}{ds}\}=-\kappa w-\kappa(\frac{(m\int^s_0\kappa ds+A)^2}{1-(m\int^s_0\kappa ds+A)^2})^{1/2}\sqrt{1-w^2-(\frac{1}{\kappa}\frac{dw}{ds})^2}.
\end{equation*}Considering the definition of the slant helix, let's find a solution of the form $\frac{1}{\kappa}\frac{d\xi}{ds}=\delta$,
this is $\xi=\delta\int^s _0\kappa ds,$ where $\delta$ is an constant and the $\xi$ function satisfies:\\
$1-\xi^2-(\frac{1}{\kappa}\frac{d\xi}{ds})^2>0$.\\
Replacing in the equation above we find that
\begin{equation*}
0=-\kappa (\delta \int^s_0 \kappa ds)+\kappa\frac{\mid(m\int^s_0\kappa ds+A)\mid}{\sqrt{1-(m\int^s_0\kappa ds+A)^2}}\sqrt{1-(\delta\int^s_0 \kappa ds)^2-\delta^2}\\
\end{equation*}If $m>0$, we take $\delta=\frac{ m }{\sqrt{1+m^2}}$, $A=0$ and we get
 \begin{eqnarray*}
 &&-\kappa (\delta \int^s_0 \kappa ds)+\kappa\frac{\mid(m\int^s_0\kappa ds)\mid}{\sqrt{1-(m\int^s_0\kappa ds)^2}}\sqrt{1-\delta^2}\sqrt{1-(\frac{\delta}{\sqrt{1-\delta^2}}\int^s_0 \kappa ds)^2}\\
&=&-\kappa (\delta \int^s_0 \kappa ds)+\kappa(m\int^s_0\kappa ds)\sqrt{1-\delta^2}
\\&=& 0,
\end{eqnarray*}
 respectively if $m<0$, we take $\delta=\frac{ -m }{\sqrt{1+m^2}}$, $A=0$ and we get
\begin{eqnarray*}
 &&-\kappa (\delta \int^s_0 \kappa ds)+\kappa\frac{\mid(m\int^s_0\kappa ds)\mid}{\sqrt{1-(m\int^s_0\kappa ds)^2}}\sqrt{1-\delta^2}\sqrt{1-(\frac{\delta}{\sqrt{1-\delta^2}}\int^s_0 \kappa ds)^2}\\
&=&-\kappa (\delta \int^s_0 \kappa ds)-\kappa(m\int^s_0\kappa ds)\sqrt{1-\delta^2}
\\&=& 0.
\end{eqnarray*}
Therefore, $\xi=\frac{m}{\sqrt{1+m^2}}\int^s _0\kappa ds,$ ($m>0$) is a solution of the non-linear second-order differential equation.
\begin{equation*}
\frac{d}{ds}\{\frac{1}{\kappa}\frac{dw}{ds}\}=-\kappa w-\kappa\frac{(m\int^s_0\kappa ds)}{\sqrt{1-(m\int^s_0\kappa ds)^2}}\sqrt{1-w^2-(\frac{1}{\kappa}\frac{dw}{ds})^2},
\end{equation*}
respectively $\xi=\frac{-m}{\sqrt{1+m^2}}\int^s _0\kappa ds,$ ($m<0$) is a solution of the non-linear second-order differential equation.
\begin{equation*}
\frac{d}{ds}\{\frac{1}{\kappa}\frac{dw}{ds}\}=-\kappa w+\kappa\frac{(m\int^s_0\kappa ds)}{\sqrt{1-(m\int^s_0\kappa ds)^2}}\sqrt{1-w^2-(\frac{1}{\kappa}\frac{dw}{ds})^2},
\end{equation*}\\Let's calculate the  position vector of $\alpha$ using the expression
\begin{eqnarray*}
x(s)&=&\int{\sqrt{1-\xi^2}\cos(\int{\frac{\kappa\sqrt{1-\xi^2-(\frac{1}{\kappa}\frac{d\xi}{ds})^2}}{1-\xi^2}ds})ds}\\
y(s)&=&\int{\sqrt{1-\xi^2}\sin(\int{\frac{\kappa\sqrt{1-\xi^2-(\frac{1}{\kappa}\frac{d\xi}{ds})^2}}{1-\xi^2}ds})ds}\\
z(s)&=&\int {\xi} ds.
\end{eqnarray*}In effect,
\begin{eqnarray*}
&&\int{\frac{\kappa\sqrt{1-\xi^2-(\frac{1}{\kappa}\frac{d\xi}{ds})^2}}{1-\xi^2}}ds=\frac{1}{\sqrt{1+m^2}}
\int{\frac{\sqrt{1+m^2(\int^s_0{\kappa}ds)^2}}{1-(\frac{m^2}{1+m^2})(\int^s_0{\kappa}ds)^2}}ds\\
&=&-\frac{\sqrt{1+m^2}}{m}\arccos{(m\int^s_0\kappa ds)}-\arctan{(\frac{m^2\int^s_0\kappa ds}{\sqrt{1+m^2}\sqrt{1-m^2(\int^s_0\kappa ds)^2}}}).
\end{eqnarray*}
Therefore, considering case $m>0$, the components of the position vector are:
\begin{eqnarray*}
x(s)&=&\int(\sqrt{1-m^2(\int^s_0 \kappa ds)^2}\cos[\frac{\sqrt{1+m^2}}{m}\arccos(m\int^s_0\kappa ds)]\\&-&\frac{m^2}{\sqrt{1+m^2}}(\int^s_0{\kappa ds})\sin[\frac{\sqrt{1+m^2}}{m}\arccos(m\int^s_0\kappa ds)])ds\\
&=&\frac{1}{\sqrt{1+m^2}}\int(\int{\sin[\frac{\sqrt{1+m^2}\arccos(m\int^s_0\kappa ds)}{m}]\kappa(s)}ds)ds\\
y(s)&=&\int(-\sqrt{1-m^2(\int^s_0 \kappa ds)^2}\sin[\frac{\sqrt{1+m^2}}{m}\arccos(m\int^s_0\kappa ds)]\\&-&\frac{m^2}{\sqrt{1+m^2}}(\int^s_0{\kappa ds})\cos[\frac{\sqrt{1+m^2}}{m}\arccos(m\int^s_0\kappa ds)])ds\\
&=&\frac{1}{\sqrt{1+m^2}}\int(\int{\cos[\frac{\sqrt{1+m^2}\arccos(m\int^s_0\kappa ds)}{m}]\kappa(s)}ds)ds\\
z(s)&=&\frac{m}{\sqrt{1+m^2}}\int(\int^s_0\kappa ds)ds.
\end{eqnarray*}
Respectively, considering the case $m<0$, the components of the position vector are:
\begin{eqnarray*}
x(s)&=&\frac{1}{\sqrt{1+m^2}}\int(\int{\sin[\frac{\sqrt{1+m^2}\arccos(m\int^s_0\kappa ds)}{m}]\kappa(s)}ds)ds\\
y(s)&=&\frac{1}{\sqrt{1+m^2}}\int(\int{\cos[\frac{\sqrt{1+m^2}\arccos(m\int^s_0\kappa ds)}{m}]\kappa(s)}ds)ds\\
z(s)&=&\frac{-m}{\sqrt{1+m^2}}\int(\int^s_0\kappa ds)ds.
\end{eqnarray*}
Now assume property (3) holds, a direct calculation shows that
\begin{equation*}
<\textbf{n}(s),(0,0,1)>=\frac{\mid m\mid}{\sqrt{1+m^2}},
\end{equation*}this is the principal normal lines of $\alpha$ make a constant angle with a fixed direction $U=(0,0,1).$
\end{proof}
\section{Conclusion.}
The differential equation
 \begin{equation*}
\frac{d}{ds}\{\frac{1}{\kappa}\frac{d\xi}{ds}\}=-\kappa \xi+\tau\sqrt{1-\xi^2-(\frac{1}{\kappa}\frac{d\xi}{ds})^2}
\end{equation*} and the curve
 $\alpha(s)=(x(s),y(s),z(s))$, where
\begin{eqnarray*}
x(s)&=&\int{\sqrt{1-\xi^2}\cos(\int{\frac{\kappa\sqrt{1-\xi^2-(\frac{1}{\kappa}\frac{d\xi}{ds})^2}}{1-\xi^2}ds})ds}\\
y(s)&=&\int{\sqrt{1-\xi^2}\sin(\int{\frac{\kappa\sqrt{1-\xi^2-(\frac{1}{\kappa}\frac{d\xi}{ds})^2}}{1-\xi^2}ds})ds}\\
z(s)&=&\int {\xi} ds,
\end{eqnarray*} associated with it, allow us to study from another point of view the curves parametrized by arc length in space.

\bibliographystyle{amsplain}

\end{document}